\documentclass[11pt, twoside]{article}
\usepackage{amsfonts,amssymb,amsmath,amsthm}
\usepackage{graphicx}
\usepackage[all]{xy}
\usepackage{multirow} 
\usepackage{psfrag,xmpmulti,amscd,color,pstricks, import}
 
 \divide\oddsidemargin by 2
\newtheorem{thm}{Theorem}[section]
\newtheorem{cor}[thm]{Corollary}
\newtheorem{lem}[thm]{Lemma}
\newtheorem{prop}[thm]{Proposition}

\theoremstyle{definition}
\newtheorem{defn}[thm]{Definition}

\newtheorem*{rem*}{Remark}

\newtheorem*{thmA}{Theorem A}
\newtheorem*{thmB}{Theorem B}
\newtheorem*{thmC}{Theorem C}
\newtheorem*{corD}{Corollary D}

\numberwithin{equation}{section}

\definecolor{OrangeRed}{cmyk}{0,0.6,1,0}            
\definecolor{DarkBlue}{cmyk}{1,1,0,0.20}
\definecolor{DarkGreen}{cmyk}{1,0,0.6,0.2}
\definecolor{myblue}{rgb}{0.66,0.78,1.00}
\definecolor{Violet}{cmyk}{0.79,0.88,0,0}
\definecolor{Lavender}{cmyk}{0,0.48,0,0}

\renewcommand{\Im}{\operatorname{Im}}
\renewcommand{\Re}{\operatorname{Re}}

\newcommand{\C}{{\mathbb C}}

\newcommand{\R}{{\mathbb R}}

\renewcommand{\epsilon}{\varepsilon}
\renewcommand{\phi}{\varphi}

\title{Dynamics inside Parabolic Basins}

\vspace{5cm}
\author{   
Mi Hu \\
\small Department of Mathematical, Physical and Computer Sciences,\\ 
\small University of Parma\\ 
\small Parma, 43124, Italy}

\begin{document}

    \maketitle  
\begin{abstract}
 In this paper, we investigate the behavior of orbits inside parabolic basins. 		Let $f(z)=z+az^{m+1}+(\text{higher terms}), m\geq1, a\neq0,$ and $\Omega_j$ be the immediate basin of $\mathcal{A}_j.$ We choose an arbitrary constant $C>0$ and a point $q\in{\bf v_j}\cap\mathcal{P}_j$. Then there exists a point $z_0\in \Omega_j$ so that for any $\tilde{q}\in Q:= (\cup_{l=0}^{\infty}f^{-l}(f^k(q)))\cap\Omega_j~ (l, k$ are non-negative integers), the Kobayashi distance $d_{\Omega_j}(z_0, \tilde{q})\geq C$, where $d_{\Omega_j}$ is the Kobayashi metric. In a previous paper \cite{RefH}, we showed that this result is not valid for attracting basins.
\end{abstract}

\section{Introduction}\label{sec1}

Let $\hat{\mathbb{C}}=\mathbb{C}\cup\{\infty\}, f: \hat{\mathbb{C}} \rightarrow \hat{\mathbb{C}}$ be a nonconstant holomorphic map, and $f^{ n}: \hat{\mathbb{C}} \rightarrow \hat{\mathbb{C}}$ be its $n$-fold iterate. In complex dynamics  \cite{RefB}, \cite{RefCG}, \cite{RefDH}, and \cite{RefM}, there are two crucial disjoint invariant sets, the {\sl Julia set}, and the {\sl Fatou set} of $f$, which partition the sphere $\hat{\mathbb{C}}$.
The Fatou set of $f$ is defined as the largest open set $\mathcal{F}$ where the family of iterates is locally normal. In other words, for any point $z\in \mathcal{F}$, there exists some neighborhood $U$ of $z$ so that the sequence of iterates of the map restricted to $U$ forms a normal family, so the orbits of iteration are well-behaved. 
The complement of the Fatou set is called the Julia set.

For $z\in \hat{\mathbb{C}}$, we call the set $\{z_n\}=\{z_1=f(z_0), z_2=f^2(z_0), \cdots\}$ the orbit of the point $z=z_0$. 
If $z_N=z_0$ for some integer $N$, we say that $z_0$ is a periodic point of $f$.
If $N=1$, then $z_0$ is a fixed point of $f.$ 

There have been many studies about probability measures that can describe the dynamics on the Julia set.
For example, if $z$ is any non-exceptional point,  the inverse orbits $\{f^{-n}(z)\}$ equidistribute toward the Green measure $\mu$ which lives on the Julia set. This was proved already by Brolin \cite{RefBrolin} in 1965, and many improvements and generalizations have been made \cite{RefDO}, \cite{RefDS}. However, this equidistribution toward $\mu$ is in the weak sense, and hence it is with respect to the Euclidean metric. Therefore, it is a reasonable question to ask how dense $\{f^{-n}(z)\}$ is in $\mathcal{F}$ near the boundary of $\mathcal{F}$ if we use finer metrics, for instance, the Kobayashi metric. The Kobayashi metric is an important tool in complex dynamics, see examples in \cite{RefAM}, \cite{RefAR}, \cite{RefBracci}, and \cite{RefBRS}.

There are some classical results about the behavior of a rational function on the Fatou set $\mathcal{F}$ as well. The connected components of the Fatou set of $f$ are called Fatou components. A Fatou component $\Omega\subset\hat{\mathbb{C}}$ of $f$ is invariant if $f(\Omega) = \Omega$. At the beginning of the $20$th century, Fatou \cite{RefFP} classified all possible invariant Fatou components of rational functions on the Riemann sphere. He proved that only three cases can occur: attracting, parabolic, and rotation. And in the '80s, Sullivan \cite{RefSD} completed the classification of Fatou components. He proved that every Fatou component of a rational map is eventually periodic, i.e., there are $n,m \in \mathbb{N}$ such that $f^{n+m}(\Omega)=f^m(\Omega)$. 
Furthermore, there are many investigations about the boundaries of the connected Fatou components. In a recent paper by Roesch and Yin \cite{RefRY}, they proved that if $f$ is a polynomial, any bounded Fatou component, which is not a Siegel disk, is a Jordan domain. 




In the paper \cite{RefH}, we investigated the behavior of orbits inside attracting basins, especially near the boundary, and obtained the following theorem:
\begin{thm}\label{the1}
	Suppose $f(z)$ is a polynomial of degree $N\geq 2$ on $\mathbb{C}$, $p$ is an attracting fixed point of $f(z),$ $\Omega_1$ is the immediate basin of attraction of $p$, $\{f^{-1}(p)\}\cap \Omega_1\neq\{p\}$,
	and $\mathcal{A}(p)$ is the basin of attraction of $p$, $\Omega_i (i=1, 2, \cdots)$ are the connected components of $\mathcal{A}(p)$. Then there is a constant $\tilde{C}$ so that for every point $z_0$ inside any $\Omega_i$, there exists a point $q\in \cup_k f^{-k}(p)$ inside $\Omega_i$ such that $d_{\Omega_i}(z_0, q)\leq \tilde{C}$, where $d_{\Omega_i}$ is the Kobayashi distance on $\Omega_i.$  
\end{thm}

For the case when $\{f^{-1}(p)\}\cap\Omega_1=\{p\}$, we proved a suitably modified version of Theorem \ref{the1} in our paper \cite{RefH}. In conclusion, we proved that there is a constant $C$ and a point $p'~ (p'=p \text{ or $p'$ is very close to $p$})$ so that for every point $z_0$ inside any $\Omega_i$, there exists a point $q\in \cup_k f^{-k}(p')$ inside $\Omega_i$ such that $d_{\Omega_i}(z_0, q)\leq C$, where $d_{\Omega_i}$ is the Kobayashi distance on $\Omega_i.$  This Theorem \ref{the1} essentially shows the shadowing of any arbitrary orbit by an orbit of one preimage of the fixed point $p.$

It is a natural question to ask if the same statement holds for all parabolic basins. 
Our main results as following in this paper show that theorem \ref{the1} in paper \cite{RefH} is no longer valid for parabolic basins:
\begin{thmA}\label{the2}
	Let $f(z)=z+z^2$. We choose an arbitrary constant $C>0$ and the point $q=-\frac{1}{2}\in\mathcal{A}$.  Then there exists a point $z_0\in \mathcal{A}$ so that for any $\tilde{q}\in Q:= \cup_{l, k=0}^{\infty}\{f^{-l}(f^k(-\frac{1}{2}))\}$
	the Kobayashi distance $d_{\mathcal {A}}(z_0, \tilde{q})\geq C$. 
\end{thmA}
Here $\mathcal{A}=\mathcal{A}(0, -1)$ in Theorem A is the parabolic basin of $f=z+z^2$ with the attraction vector ${\bf v}=-1$, see Definitions \ref{def1} and \ref{def3}.

 We also generalize Theorem A to the case of several petals inside the parabolic basin in Theorem B:
 \begin{thmB}\label{the3}
 	Let $f(z)=z+az^{m+1}, m\geq1, a\neq0$, and $\Omega_j$ be the immediate basin of $\mathcal{A}_j.$  
 	We choose an arbitrary constant $C>0$ and a point $q\in{\bf v_j}\cap\mathcal{P}_j$. Then there exists a point $z_0\in \Omega_j$ so that for any $\tilde{q}\in Q:= (\cup_{l=0}^{\infty}f^{-l}(f^k(q)))\cap\Omega_j~(l, k$ are non-negative integers), the Kobayashi distance $d_{\Omega_j}(z_0, \tilde{q})\geq C$, where $d_{\Omega_j}$ is the Kobayashi metric. 
 \end{thmB}
 
 Here $ {\bf v_j}$ is an attraction vector in the tangent space of $\C$ at $0$,  $\mathcal{P}_j$ is an {\em  attracting petal} for $f$ for the vector ${\bf v_j}$ at $0$, $\mathcal{A}_j=\mathcal{A}(0, {\bf v_j})$ is the parabolic basin of attraction associated to ${\bf v_j},$ see Definitions \ref{def1},  \ref{def3} and \ref{def4}.
 
 In the end, we sketch how to handle the behavior of orbits inside parabolic basins of general polynomials.
 \begin{thmC}\label{the4}
 	Let $f(z)=z+az^{m+1}+(\text{higher terms}), m\geq1, a\neq0,$ and $\Omega_j$ be the immediate basin of $\mathcal{A}_j.$ We choose an arbitrary constant $C>0$ and a point $q\in{\bf v_j}\cap\mathcal{P}_j$. Then there exists a point $z_0\in \Omega_j$ so that for any $\tilde{q}\in Q:= (\cup_{l=0}^{\infty}f^{-l}(f^k(q)))\cap\Omega_j~ (l, k$ are non-negative integers), the Kobayashi distance $d_{\Omega_j}(z_0, \tilde{q})\geq C$, where $d_{\Omega_j}$ is the Kobayashi metric. 
 \end{thmC} 

 Furthermore, using the distance decreasing property (see Proposition \ref{pro2}) of the Kobayashi metric, Theorem C implies the following Corollary D, which says the set of all $z_0$ is dense in the boundary of the parabolic basin of $f$:
 \begin{corD}\label{co1}
 	Let $\Omega_{i,j}$ be the connected components of $\mathcal{A}_j$. Let $X\subset\mathcal{A}_j$ be the set of all $z_0\in\mathcal{A}_j$ so that $d_{\Omega_{i,j}}(z_0, \tilde{q})\geq C$ for any $\tilde{q}$ in the same connected component $\Omega_{i,j}$ as $z_0$. If $z\in X,$ then any point $w\in f^{-1}(z)$ is in $X.$ Hence $X$ is dense in the boundary of $\mathcal{A}_j$.
 \end{corD}

This paper is organized as follows: in section 2, we recall some definitions and results \cite{RefM} about holomorphic dynamics of polynomials in a neighborhood of the parabolic fixed point and the Kobayashi metric; in section 3, we prove our main results, Theorem A, Theorem B, and Theorem C.

\section*{Acknowledgement}
I appreciate my advisor John Erik Forn\ae ss very much for giving me this research problem and for his significant comments and patient guidance. In addition, I am very grateful for all support from the mathematical department during my visit to NTNU in Norway so that this research can work well.

\section{Preliminary}
	\subsection{Holomorphic dynamics of polynomials in a neighborhood of the parabolic fixed point. }
Let us first recall some definitions and results \cite{RefM} about holomorphic dynamics of a polynomial $f=z+a_2z^2+a_3z^3+\cdots$ in a neighborhood of the parabolic fixed point $0$. 
\begin{defn}\label{def1}
	Let $f(z)=z+az^{m+1}+(\text{higher terms}), m\geq1, a\neq0.$ A complex number ${\bf v}$ will be called an {\em attraction vector} at the origin if $ma{\bf v}^m=-1, $ and  a {\em repulsion vector} at the origin if $ma{\bf v}^m=1. $  Note here that ${\bf v}$ should be thought of as a tangent vector at the origin.
	We say that some orbit $z_0\mapsto z_1 \mapsto z_2\mapsto\cdots$ for the map $f$ converges to zero {\em nontrivially} if $z_n\rightarrow0$ as $n\rightarrow\infty,$ but no $z_n$ is actually equal to zero. 
	There are $m$ equally spaced attraction vectors at the origin.       
\end{defn}

\begin{lem}\label{lem12}
	If an orbit $f: z_0\mapsto z_1\mapsto\cdots$ converges to zero nontrivially, then $z_k$ is asymptotic to ${\bf v_j}/\sqrt[m]{k}$ as $k\rightarrow +\infty$ for one of the $m$ attraction vectors ${\bf v_j}.$ 
\end{lem}
\begin{proof}
	See the proof in chapter 10 of Milnor's book \cite{RefM}.
\end{proof}

\begin{defn}\label{def2}
	If an orbit $z_0\mapsto z_1\mapsto\cdots$ under $f$ converges to zero with $z_k\sim {\bf v_j}/\sqrt[m]{k}$, we will say that this orbit $\{z_k\}$ tends to zero from the direction ${\bf v_j}.$ 
\end{defn}

\begin{defn}\label{def3}
	Given an attraction vector $ {\bf v_j}$ in the tangent space of $\C$ at $0$, the associated {\em parabolic basin of attraction} $\mathcal{A}_j=\mathcal{A}(0, {\bf v_j})$ is defined to be the set consisting of all $z_0\in\C$ for which the orbit $z_0\mapsto z_1\mapsto\cdots$ converges to $0$ from the direction ${\bf v_j}$. 
	
\end{defn}  

\begin{defn} \label{def4}
	Suppose $f$ is defined and univalent on some neighborhood $N\in\mathbb{C}.$ An open set $\mathcal{P}_j\subset N$ is called an {\em  attracting petal} for $f$ for the vector ${\bf v_j}$ at $0$ if 
	
	(1) $f$ maps $\mathcal{P}_j$ into itself, and
	
	(2) an orbit $z_0\mapsto z_1\mapsto\cdots$ under $f$ is eventually absorbed by $\mathcal{P}_j$ if and only if it converges to $0$ from the direction ${\bf v_j}.$   
\end{defn}

\subsection{The Kobayashi metric}\label{subsec1}
\begin{defn}\label{def5}
	Let $\hat{\Omega}\subset\mathbb C$ be a domain. We choose a point $z\in \hat{\Omega}$ and a vector $\xi$ tangent to the plane at the point $z.$ Let $\triangle$ denote the unit disk in the complex plane.
	We define the {\em Kobayashi metric} 
	$$
	F_{\hat{\Omega}}(z, \xi):=\inf\{\lambda>0 : \exists f: \triangle\stackrel{hol}{\longrightarrow} \hat{\Omega}, f(0)=z, \lambda f'(0)=\xi\}.
	$$

	Let $\gamma: [0, 1]\rightarrow \hat{\Omega}$ be a piecewise smooth curve.
	The {\em Kobayashi length} of $\gamma$ is defined to be 
	$$ L_{\hat{\Omega}} (\gamma)=\int_{\gamma} F_{\hat{\Omega}}(z, \xi) \lvert dz\rvert=\int_{0}^{1}F_{\hat{\Omega}}\big(\gamma(t), \gamma'(t)\big)\lvert \gamma'(t)\rvert dt.$$

	For any two points $z_1$ and $z_2$ in $\hat{\Omega}$, the {\em Kobayashi distance} between $z_1$ and $z_2$ is defined to be 
	$$d_{\hat{\Omega}}(z_1, z_2)=\inf\{L_{\hat{\Omega}} (\gamma): \gamma ~ \text{is a piecewise smooth curve connecting} ~z_1~ \text{and} ~z_2 \}.$$

	Note that $d_{\hat{\Omega}}(z_1, z_2)$ is defined when $z_1$ and $z_2$ are in the same connected component of $\hat{\Omega}.$

	
\end{defn}


\begin{prop}[The distance decreasing property of the Kobayashi Metric \cite{RefK}]\label{pro2}
	Suppose $\Omega_1$ and $\Omega_2$ are domains in $\mathbb C$, $z, \omega\in \Omega_1, \xi\in\mathbb C,$ and $f:\Omega_1\rightarrow\Omega_2$ is holomorphic. Then 
	$$F_{\Omega_2}(f(z
	), f'(z)\xi)\leq F_{\Omega_1}(z, \xi), ~~~d_{\Omega_2}(f(z), f(\omega))\leq d_{\Omega_1}(z,\omega).$$
\end{prop}

\begin{cor}\label{cor3}
	Suppose $\Omega_1\subseteq\Omega_2\subseteq\mathbb C.$ Then for any $z, \omega\in \Omega_1$ and $\xi\in\mathbb C,$ we have 
	$$F_{\Omega_2}(z, \xi)\leq F_{\Omega_1}(z, \xi), ~~~~d_{\Omega_2}(z, \omega)\leq d_{\Omega_1}(z,\omega).$$
\end{cor}

\section{ Proof of the main theorems}
In this section, we will prove our main theorems: Theorem A, Theorem B, and Theorem C.
\subsection{Dynamics inside the parabolic basin of $f(z)=z+z^2$}\label{subsec2}
Let us recall the statement of our main Theorem A:
%
\begin{thmA}\label{the}
	Let $f(z)=z+z^2$. We choose an arbitrary constant $C>0$ and the point $q=-\frac{1}{2}\in\mathcal{A}$.  Then there exists a point $z_0\in \mathcal{A}$ so that for any $\tilde{q}\in Q:= \cup_{l, k=0}^{\infty}\{f^{-l}(f^k(-\frac{1}{2}))\}$
	the Kobayashi distance $d_{\mathcal {A}}(z_0, \tilde{q})\geq C$. 
\end{thmA}


As pointed out above, this is opposite to Theorem \ref{the1} in the paper \cite{RefH} for attracting basins. 

\begin{proof}
	Let $\R ^+=[0,\infty)$ and $\R ^-=(-\infty, 0]$ be the positive and negative real axis, respectively. Then the parabolic basin $\mathcal{A}\subsetneqq(\mathbb C\setminus\R ^+).$
	Let $\mathbb H$ be the upper half-plane, 
	$\varphi(z): \mathbb C\setminus\mathbb R^+\rightarrow \mathbb H$ with $\varphi(z)=\sqrt{z}$ (see Figure 1). We know the Kobayashi metric on the upper half plane is $F_{\mathbb H}=\frac{|d\omega|}{\Im \omega}$ for $\omega\in\mathbb H.$ Hence the Kobayashi metric on $\mathbb C\setminus\R^+$ is $F_{\mathbb C\setminus\R^+}=\frac{|\varphi'|}{\Im \varphi}|dz|=\frac{1}{2|\sqrt{z}|\Im \sqrt{z}}|dz|=\frac{1}{2r\sin\frac{\theta}{2}}|dz|$ for $z=re^{i\theta}\in\mathbb C\setminus\R^+, i.e., \theta\in(0, 2\pi).$

	\begin{figure}[!htb]
		\centering 
		\includegraphics[width=0.75\textwidth,height=0.15\textheight]{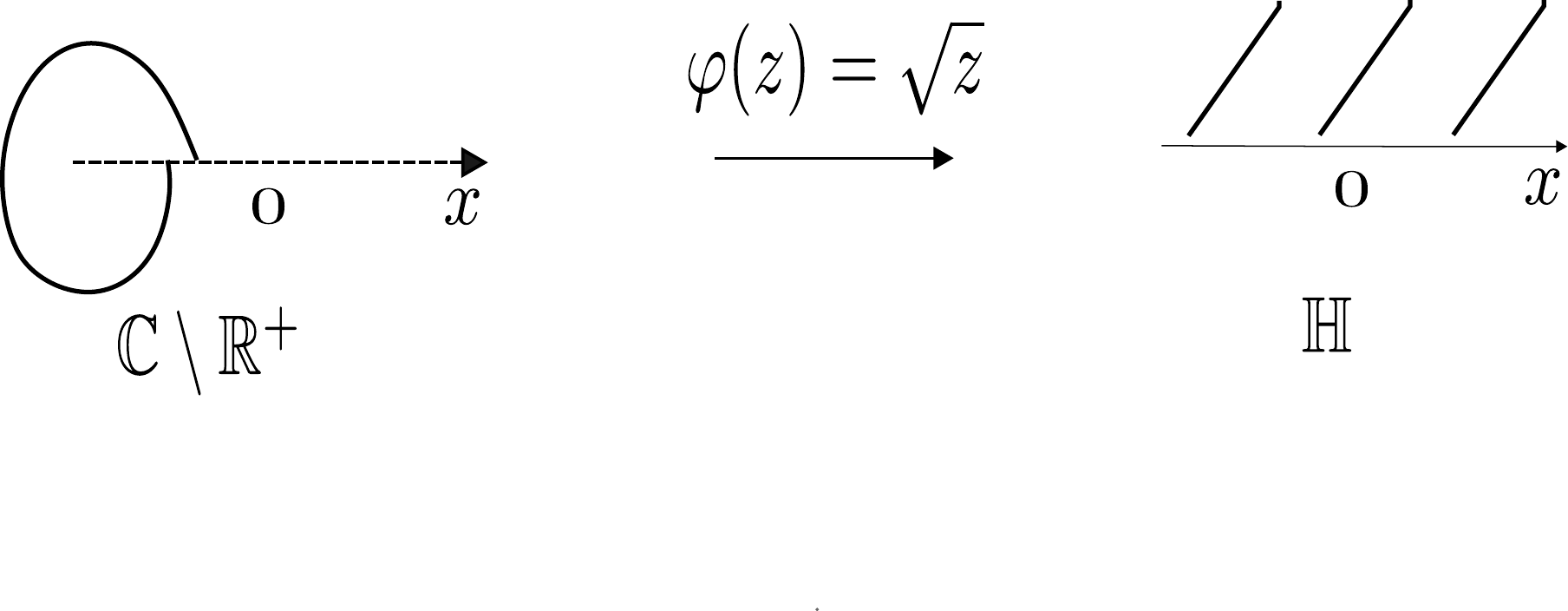}
		\caption{The map $\phi$}
		\label{Figure1}
	\end{figure}

	We know that $f^k(q)$ will always be on the negative real axis for any integer $k\geq0$. 
	We choose $\theta_0>0$ and two rays $l_1:=\{z=re^{i\theta_0}, r>0\}, l_2:=\{z=re^{-i\theta_0}, r>0\}$. Then we denote by
 $T:=\{z=re^{i\theta}, r>0, \theta_0\leq\theta\leq2\pi-\theta_0\}$, a sector inside $\mathbb{C}\setminus\R^+$. 
	
	Before continuing with the proof of Theorem A, we have the following well-known Lemma \ref{lem4}. For the reader's convenience and to introduce notation, we include the proof and define a Left/Right Pac-Man for an easy explanation of the proof.

	\begin{defn} \label{def6}
We call a domain $D_{R}: =\{z=r e^{i\theta}, 0<r\leq R,  \theta_0<\theta<2\pi-\theta_0\}$ 
 a {\em Left Pac-Man} and 
 $\tilde{D}_{R}: =\{z=r e^{i\theta}, 0<r\leq R, -\pi+ \theta_0<\theta<\pi-\theta_0\}$ 
 a {\em Right Pac-Man}.
	\end{defn}

	\begin{lem}\label{lem4}
		There exists a Left Pac-Man $D_{R_0}$ such that $D_{R_0}\subsetneqq T\cap\mathcal{A}.$ 
	\end{lem} 
	
	\begin{proof}
		We want to know how orbits go precisely near the parabolic fixed point at $0.$ Let $\omega=\phi(z)=-1/z$ send $0$ to $\infty,$ then the conjugated map has the expansion $$F(\omega)=\phi\circ f\circ\phi^{-1}(\omega)=\omega+1+o(1) ~~ \text{as} ~~ |\omega|\rightarrow\infty.$$
		And we have $l_1(l_2)$ is mapped to two new rays $l^\omega_1:=\{z=re^{i(\pi-\theta_0)}, r>0\}(l^\omega_2:=\{z=re^{i(-\pi+\theta_0)}\})$; 
		$T$ is mapped to $T'=\{r e^{i\theta}, r>0,-\pi+\theta_0\leq\theta\leq\pi-\theta_0\}$;
		the Left Pac-Man $D_{R}$ is mapped to $T'\setminus \tilde{D}_{\frac{1}{R}}$ for any radius $R$ (see Figure \ref{Figure2}).
\begin{figure}[!htb]
	\centering 
	\includegraphics[width=0.75\textwidth,height=0.25\textheight]{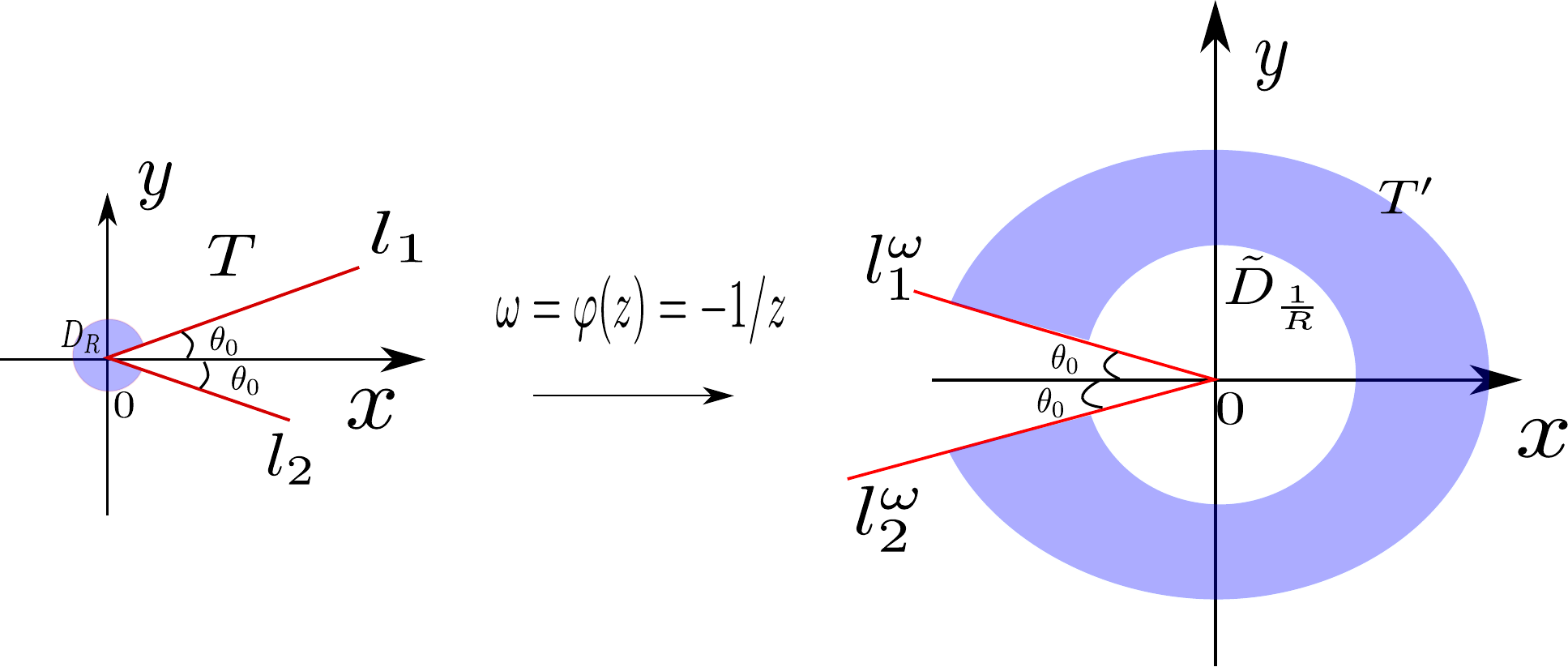}
	\caption{The image of a Left  Pac-Man}
	\label{Figure2}
\end{figure}

		We choose a Right Pac-Man $\tilde{D}_{\frac{1}{r_0}}$ such that for any $\omega \in T'\setminus \tilde{D}_{\frac{1}{r_0}},$ we have $|o(1)|< \frac{\theta_0}{3}$. Note here, actually, $r_0=r_0(\theta_0)$ should be sufficiently small.  
		Then we draw the upper tangent line $L_0$ of $\tilde{D}_{\frac{1}{r_0}}$ such that the angle  between $L_0$ and the real axis is $\frac{\theta_0}{2}$. Then $L_0$ will intersect $l_1^\omega$ and the real axis, we denote these two intersect points $A_0$ and $B_0,$ respectively. Let $\frac{1}{r}=\max\{|OA_0|, |OB_0|\}>\frac{1}{r_0}$, here $O$ is the origin zero, then we choose the Right Pac-Man $\tilde{D}_{\frac{1}{r}}$ (see Figure \ref{Figure3}).
		\begin{figure}[!htb]
			\centering 
			\includegraphics[width=0.4\textwidth,height=0.3\textheight]{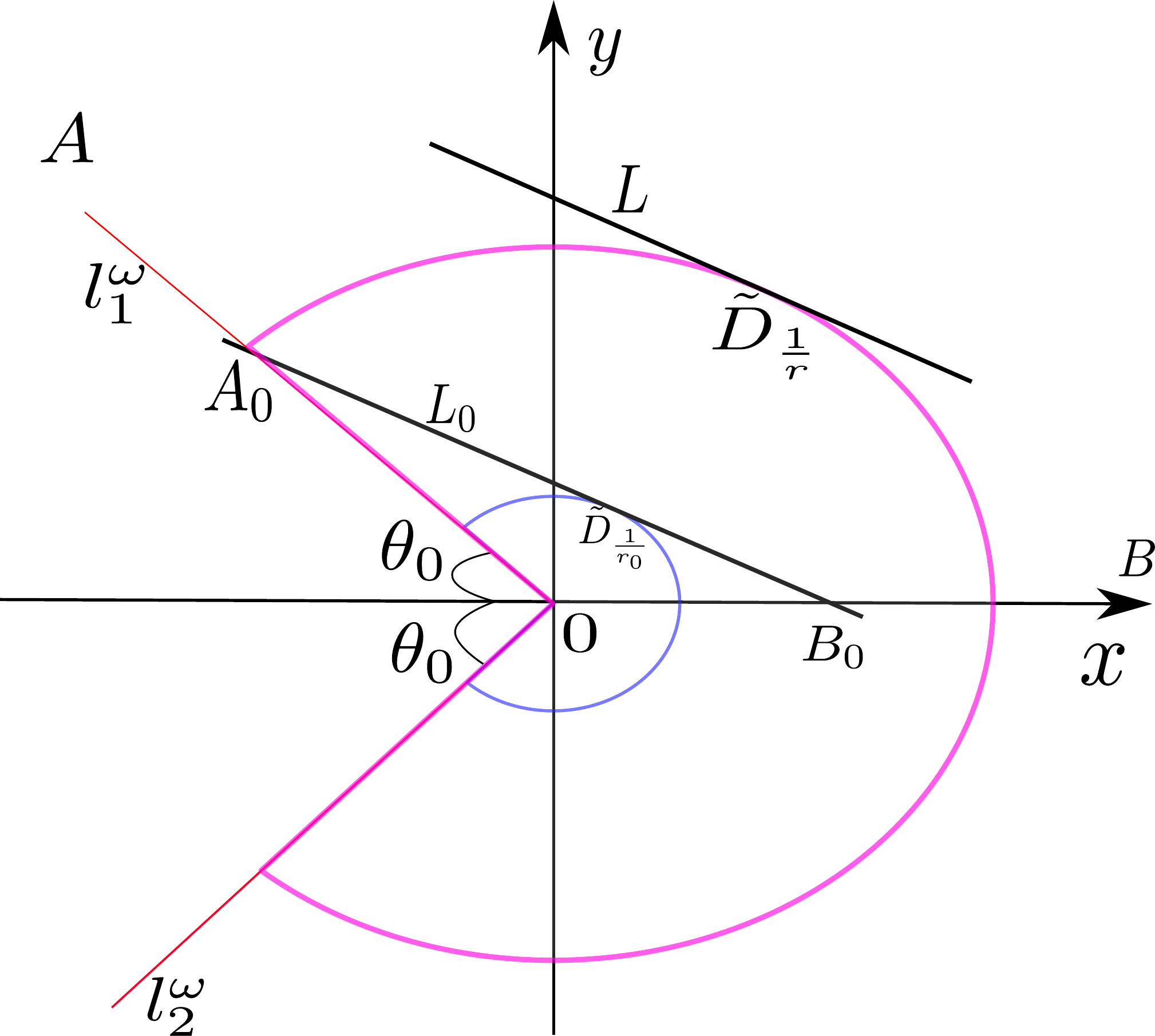}
			\caption{ The choice of the Right Pac-Man}
			\label{Figure3}
		\end{figure} 
	
		If we take any $\omega_0 \in T'\setminus \tilde{D}_{\frac{1}{r}}$, $F^n(\omega_0)=\omega_0+n+o(1)$ for all positive integers $n$ such that $|F^n(\omega_0)|\geq|\omega_0+n-n\frac{\theta_0}{3}|$,
		then we know that $F^n(\omega_0)$ will never go inside $\tilde{D}_{\frac{1}{r_0}}.$ 
		
		Therefore, let $R_0=r, \tilde{R}_0=r_0,$ then we have for any $z_0\in D_{R_0}, f^n(z_0)\rightarrow 0,$ hence $D_{R_0}\subsetneqq T\cap\mathcal{A},$ and we also know that $f^n(D_{R_0})\subseteq D_{\tilde{R}_0}.$ 

	\end{proof}

	\begin{lem}\label{lem5}
		We can choose a Left Pac-Man $D_{R'_0}\subseteq D_{R_0}$ such that $f^n(D_{R'_0})\subseteq D_{R_0}$ for all $n=1, 2,\dots$  
			\end{lem}

		\begin{proof}
			On the procedure for proving Lemma \ref{lem4}, we can draw another upper tangent line $L$ of $\tilde{D}_{R_0}$ such that the angle between $L$ and the real axis is $\frac{\theta_0}{2}$. Then $L$ will also intersect $l_1^\omega$ and the real axis, we denote these two intersect points $A$ and $B$ respectively. Let $\frac{1}{r'}=\max\{|OA|, |OB|\}$, then we choose the Right Pac-Man $\tilde{D}_{\frac{1}{r'}}.$	
			If we take any $\omega \in T'\setminus \tilde{D}_{\frac{1}{r'}}$, we know that $F^n(\omega)$ will never go inside of $\tilde{D}_{\frac{1}{r}}.$ 
			Hence, let $R'_0=\frac{1}{r'},$ we have $f^n(D_{R'_0})\subseteq D_{R_0}.$

		\end{proof}

	
	We continue with the proof of Theorem A. The idea of the proof is to find a point $z_0\in D_\epsilon\setminus\R^-$ such that for any $\tilde{q}\in Q,$ we have $d_{\mathcal {A}}(z_0, \tilde{q})\geq C$. 
	
	Now, we will consider the following cases of $\tilde{q}$ inside three subsets of $\mathcal{A}$ (see Figure \ref{Figure5}):

	Case 1: $\tilde{q}\in \big((T\cap \mathcal{A})\setminus D_{R'_0}\big)\setminus\R^-.$ Let $d^1_{\mathcal{A}}(z_0, \tilde{z})$ be the Kobayashi distance between $z_0$ and any point $\tilde{z}\in \partial D_{R'_0}\setminus R^-$ (see the blue curve in Figure \ref{Figure5}). Then we prove that $d^1_{\mathcal{A}}(z_0, \tilde{z})\geq C$.

	Case 2: $\tilde{q} \in\R^-.$ Let $d^2_{\mathcal{A}}(z_0, z')$ be the Kobayashi distance between $z_0$ and any point $z'\in \R^-\cap\mathcal{A}$ (see the pink curve in Figure \ref{Figure5}). Then we prove that $d^2_{\mathcal{A}}(z_0, z')\geq C$.
	
	Case 3: $\tilde{q} \in \mathcal{A}\cap \big\{S':=\{z=re^{i\theta}, r>0, 0<\theta\leq\theta_0\}\big\}.$
	 Let $d^3_{\mathcal{A}} (z_0, \hat{z})$ be the Kobayashi distance between $z_0$ and any point $\hat{z}\in l_1$ (see the green curve in Figure \ref{Figure5}). Then we prove that $d^3_{\mathcal{A}}(z_0, \hat{z})\geq C$.
\begin{figure}[!htb]
	\centering 
	\includegraphics[width=0.6\textwidth,height=0.4\textheight]{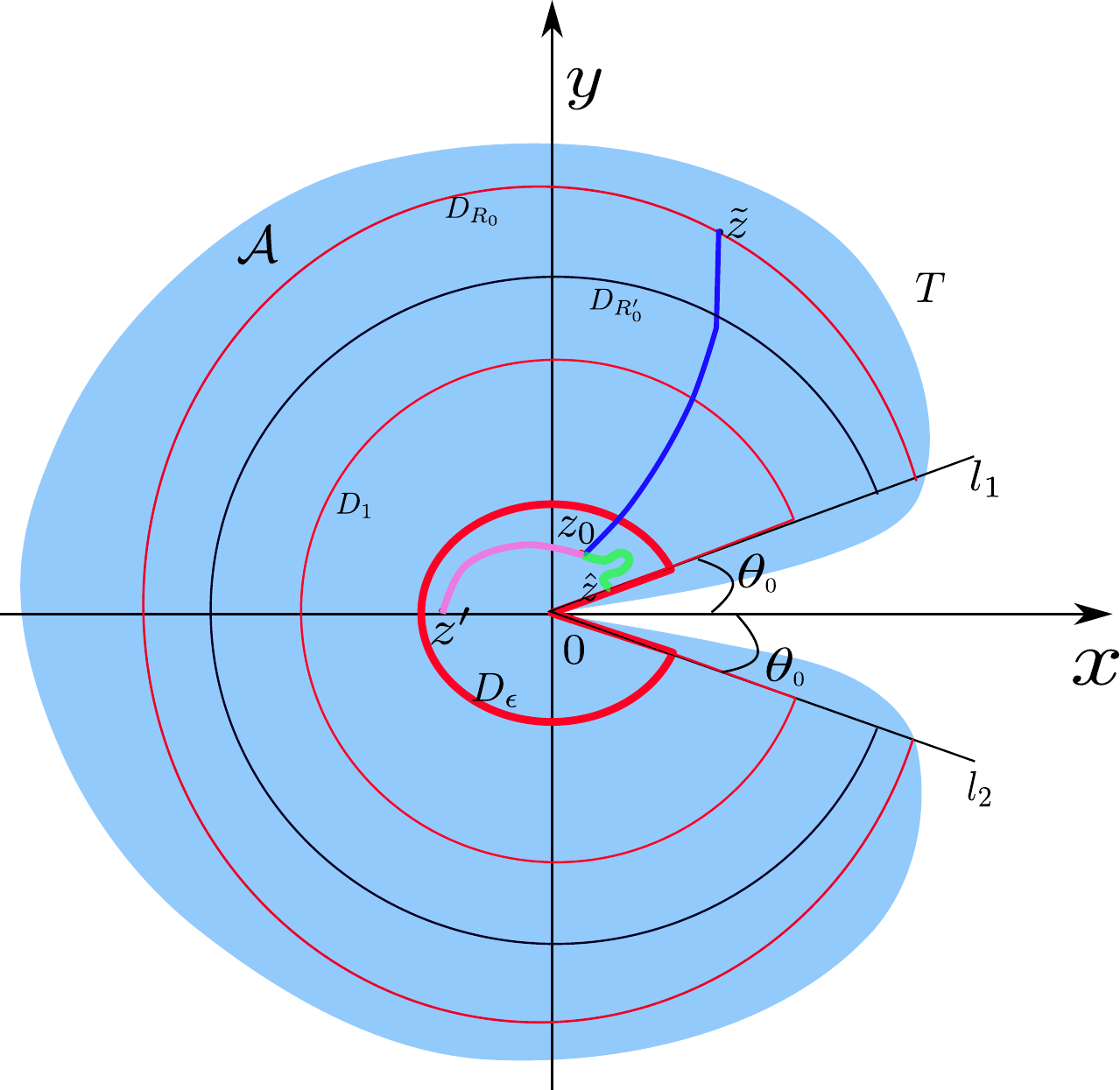}
	\caption{The three Kobayashi distances in $\mathcal{A}$}
	\label{Figure5}
\end{figure} 

{\bf Remark: }

1) In the investigation of these three cases, it will become clear how small $\theta_0$ needs to be.

	2) $\partial D_{R'_0}$ means the boundary of the Left Pac-Man $D_{R'_0}.$ It is the circular curve of $D_{R'_0},$ not including the mouth of $D_{R'_0}$ which belongs to $l_1$ and $l_2;$
	
	3) $\tilde{q}$ can never be inside $D_{R'_0}\setminus\R^-:$  If $\tilde{q} \in D_{R'_0}\setminus\R^-,$ more precisely, suppose $f^{-l}(f^k(q))\in D_{R'_0}\setminus\R^-$ for some integers $l, k\geq0,$ we iterate $l$ times of $f^{-l}(f^k(q))$, we have $f^k(q)\in D_{R_0}\setminus\R^-$ since we know $f^n(D_{R'_0})\subseteq D_{R_0}$ for any positive integer $n$ by Lemma \ref{lem4} and $f^k(q)\notin\R^-$ since $R'_0<R_0=r<r_0$, and $f$ is biholomorphic near $0.$ Hence $\Im(f^k(q))\neq0$. However, this  contradicts $\Im (f^k(q))=\Im(f^k(-\frac{1}{2}))=0.$ Thus $\tilde{q}\notin D_{R'_0}\setminus\R^-;$ 
	
	4)  if $\tilde{q}\in \mathcal{A}\setminus D_{R'_0},$ $\tilde{q}$ is far away from $\partial D_{R'_0}$, then we know that $d_{\mathcal{A}}(z_0, \tilde{q})\geq d^1_{\mathcal{A}}(z_0, \tilde{z})$. This is because  $d^1_{\mathcal{A}}(z_0, \tilde{z})$ is the minimum distance between $z_0$ and any $\tilde{q}\in \big((T\cap \mathcal{A})\setminus D_{R'_0}\big)\setminus\R^-.$

	Hence, we need to prove that we can choose $z_0$ so that all these three Kobayashi distances	
	$d^1_{\mathcal{A}}(z_0, \tilde{z}), d^2_{\mathcal{A}}(z_0, z'),  d^3_{\mathcal{A}}(z_0, \hat{z})\geq C$ for the given constant $C$. Next, we will estimate these three Kobayashi distances.

	First, we estimate $d^1_{\mathcal{A}}(z_0, \tilde{z}).$ Suppose $z_0\in D_\epsilon,$ 
	\begin{equation*}
		\begin{aligned}
			d^1_{\mathcal{A}}(z_0, \tilde{z})
			\geq &d_{\mathbb C\setminus\R^+}(z_0, \tilde{z}) \\
			=&\inf\int_{\gamma(t)}F_{\mathbb C\setminus\R^+}(\gamma(t))|\gamma'(t)|dt\\
			=&\inf\int_{\gamma(t)}\frac{1}{2|\gamma(t)|\sin\frac{\arg(\gamma(t))}{2}}|\gamma'(t)|dt\\
			\geq& \inf\int_{\gamma(t)}\frac{1}{2|\gamma(t)|}|\gamma'(t)|dt\\
			=&\frac{1}{2}\inf\int_{\epsilon}^{R'_0}\frac{|dr|}{r}\\
			\geq&\frac{1}{2}(\ln R'_0-\ln \epsilon),
		\end{aligned}	
	\end{equation*}
	where $\gamma(t)$ is a smooth path joining $z_0$ to $\tilde{z}$. The last inequality holds since there might have some derivatives of the path $\gamma(t)$ are negative in some pieces.
	In addition, we can see that 
	$ d^1_{\mathcal{A}}(z_0, \tilde{z})\rightarrow\infty$ as $\epsilon\rightarrow0.$

	Second, we estimate $d^2_{\mathcal{A}}(z_0, z')$. Let $\epsilon=\epsilon_0$, i.e., fix $\epsilon,$ and let $D_1\subset T$ be a scaling of $D_{\epsilon_0}$ by $S(z)=\frac{z}{|z_0|},$ sending $z_0, z'$ to $\tilde{z}_0:=\frac{z_0}{|z_0|}, \frac{z'}{|z_0|}$, respectively. By homogeneity, we know the Kobayashi distance  $d_{\mathbb{C}\setminus\R^+}(z_0, z')=d_{\mathbb{C}\setminus\R^+}(\tilde{z}_0, z'/|z_0|)$. Since we hope to prove $d^2_{\mathcal{A}}(z_0, z')\geq C$, we need $z_0$ to be far from $\mathbb{R}^-$, and so does $\tilde{z}_0$. Let $S_T:=\{z=e^{i\theta},  \theta_0<\theta<\frac{\pi}{2}-\theta_0\}$, assume $\tilde{z}_0\in S_T$ and $\Re \tilde{z}_0>\frac{1}{2}$, then any curve from $\tilde{z}_0$ to $\frac{z'}{|z_0|}$ must pass through a point $\tilde{z}'$ on the positive imaginary axis, i.e., $\Re \tilde{z}'=0$.  
	For simplicity, we assume this curve and $\tilde{z}'$ lie in the upper half plane. Hence $d^2_{\mathcal{A}}(z_0,   z')\geq d_{\mathbb{C}\setminus\R^+}(z_0, z')=d_{\mathbb{C}\setminus\R^+}(\tilde{z}_0, z'/|z_0|)\geq d_{\mathbb{C}\setminus\R^+}(\tilde{z}_0, \tilde{z}').$ 
	We have
	\begin{equation*}
		\begin{aligned}
			d^2_{\mathcal{A}}(z_0, z')
			&\geq d_{\mathbb C\setminus\R^+}(\tilde{z}_0, \tilde{z}')\\
			&=\inf\int_{\tilde{z}_0}^{\tilde{z}'}F_{\mathbb C\setminus\R^+}(z)\\
			&=\inf\int_{\tilde{z}_0}^{\tilde{z}'}\frac{|dz|}{|z|2\sin(\theta/2)}\\
			&\geq 	\sqrt{2}\inf\int_{\tilde{z}_0}^{\tilde{z}'}\frac{|dz|}{|z|\sin\theta}\\
			&\geq 	\sqrt{2}\inf\int_{\tilde{z}_0}^{\tilde{z}'}\frac{|dz|}{\Im z}\\
			&>	\sqrt{2}|\ln (\Im \tilde{z}')-\ln (\Im \tilde{z}_0)|.\\
		\end{aligned}
	\end{equation*}
	
Then there are three situations for the Kobayashi distance between $\tilde{z}_0$ and $\tilde{z}':$	

1) If $ \Im \tilde{z}'\geq e^C|\Im \tilde{z}_0|$ or $\Im \tilde{z}'\leq \frac{|\Im \tilde{z}_0|}{e^C}$ for the constant $C>0$ (see the blue curves on Figure \ref{Figure4}), then $|\ln (\Im \tilde{z}')-\ln (\Im \tilde{z}_0)|\geq C,$ hence $d^2_{\mathcal{A}}(z_0, z')\geq C$ is true. 
\begin{figure}[!htb]
	\centering 
	\includegraphics[width=0.4\textwidth,height=0.2\textheight]{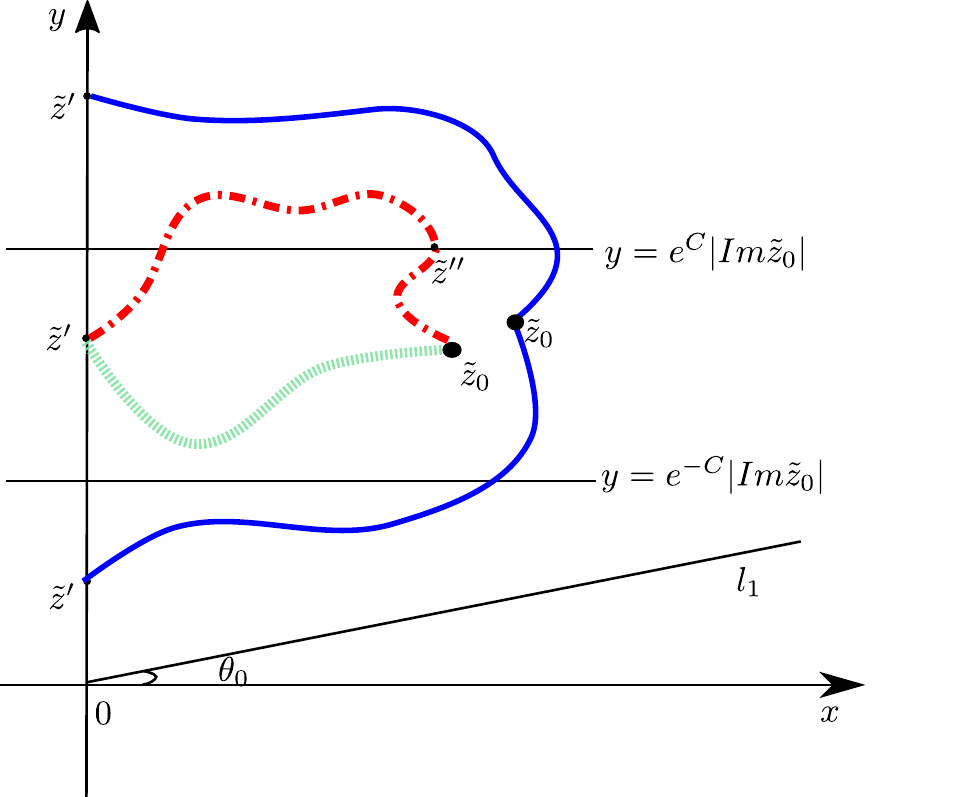}
	\caption{Three situations for the Kobayashi distance between $\tilde{z}_0$ and $\tilde{z}'$	 }
	\label{Figure4}
\end{figure}

2) lf  $\tilde{z}'\in\mathcal{L}:=\{z=x+iy, \frac{|\Im \tilde{z}_0|}{e^C}<y<e^C|\Im \tilde{z}_0|\}$ (see the green curve on Figure \ref{Figure4}). We prove that $d_{\mathcal{A}}(\tilde{z}_0, \tilde{z}')\geq \sqrt{2}\inf\int_{\tilde{z}_0}^{\tilde{z}'}\frac{|dz|}{\Im z} \geq C$ for $z\in\mathcal{L}$. 

	Let $z=x+iy\in \mathcal{L}$, 
	then $\Im z\leq e^C|\Im \tilde{z}_0|$. Hence we have  
	$$\int_{\tilde{z}_0}^{\tilde{z}'}\frac{|dz|}{\Im z}
	\geq\int_{\tilde{z}_0}^{\tilde{z}'} \frac{|dx|}{\Im z}\geq \int_{\tilde{z}_0}^{\tilde{z}'}\frac{|dx|}{e^C|\Im \tilde{z}_0| }=\frac{\Re \tilde{z}_0}{e^C}\frac{1}{|\Im \tilde{z}_0|}>\frac{1}{2e^C |\Im \tilde{z}_0|}.$$
	Note that the last inequality holds because we choose $\Re  \tilde{z}_0>1/2.$
	
	As long as $\theta_0<\frac{1}{2Ce^C}$, we can choose $\tilde{z}_0$ so that $\Im \tilde{z}_0<\frac{1}{2Ce^C}$, hence $\int_{\tilde{z}_0}^{\tilde{z}'}\frac{|dz|}{\Im z}>C.$
	
	3) If $\tilde{z}'\in  \mathcal{L}, $ but the curve $\gamma$ between $\tilde{z}_0$ and $\tilde{z}'$ gets outside of $\mathcal{L}$ starting at some point $\tilde{z}''\in \gamma\cap\mathcal{L}$  for a while (see the red curve in Figure \ref{Figure4}), then enter back to $\mathcal{L}$ again, then we still have $d^2_{\mathcal{A}}(z_0, z')\geq C$ is true because $d_{\mathbb{C}\setminus\R^+}(z_0, z')\geq d_{\mathbb{C}\setminus\R^+}(\tilde{z}_0, \tilde{z}'')\geq C.$ The last inequality holds since 1) is valid.

After these calculations, we fix $z_0$ so that $d^1_\mathcal{A}(z_0, \tilde{z})$ and $d^2_{\mathcal{A}}(z_0, z')$ are both bigger or equal to $C.$ To obtain $d^3_{\mathcal{A}}(z_0, \hat{z})\geq C,$ we will need an even smaller $\theta'_0,$ see the following calculation.
	
	We know  $d^3_{\mathcal{A}}(z_0, \hat{z})\geq d_{\mathbb{C}\setminus\R^+}(z_0, \hat{z})=d_{\mathbb{C}\setminus\R^+}(\tilde{z}_0, \hat{z}/|z_0|).$ Hence to have $d^3_{\mathcal{A}}(z_0, \hat{z})\geq C,$ we need to estimate $d_{\mathbb{C}\setminus\R^+}(\tilde{z}_0, \hat{z}/|z_0|).$

	When we estimate $d^2_{\mathcal{A}}(z_0, z')$, we send $z_0$ to $\tilde{z}_0$ and choose $\Re \tilde{z}_0>1/2$ and $\tilde{z}$ close to $l_1.$ 
	Hence $d_{\mathbb{C}\setminus\R^+}(\tilde{z}_0, \hat{z}/|z_0|)$ might be very small. To handle this situation, we first choose a disk $\Delta^K_{\mathbb{C}\setminus\R^+} (\tilde{z}_0, C)$ centered at $\tilde{z}_0$ with Kobayashi radius $C$. 
	Since this disk $\Delta^K_{\mathbb{C}\setminus\R^+} (\tilde{z}_0, C)$ is a compact subset of $\mathbb{C}\setminus\R^+$, there exists a sector $S'':=\{z=re^{i\theta}, r>0, 0<\theta<\theta'_0\}$ such that $S''\cap\Delta^K_{\mathbb{C}\setminus\R^+} (\tilde{z}_0, C)=\emptyset$, here we can assume that $\theta'_0<\theta_0.$ 
	Therefore, $d_{\mathbb{C}\setminus\R^+}(\tilde{z}_0, \hat{z}/|z_0|)\geq C$ for any $\hat{z}/|z_0|\in \mathcal{A}\cap S''.$ 
	Furthermore, for this new $\theta'_0,$ it does not change the conclusion of the estimation of $d^1_{\mathcal{A}}(z_0, \tilde{z})$ and $d^2_{\mathcal{A}}(z_0, z').$

	Therefore, there is a point $z_0$ such that all these three distances  	$d^1_{\mathcal{A}}(z_0, \tilde{z}), d^2_{\mathcal{A}}(z_0, z'),  d^3_{\mathcal{A}}(z_0, \hat{z})\geq C$ for the given constant $C.$  
	
\end{proof}

\subsection{Dynamics inside the parabolic basin of $f(z)=z+az^{m+1}, m\geq1, a\neq0$}\label{subsec2}
In this subsection, we generalize Theorem A to the case of several petals inside the parabolic basin. Let us recall the statement of our main Theorem B:
\begin{thmB}\label{the6}
Let $f(z)=z+az^{m+1}, m\geq1, a\neq0$, and $\Omega_j$ be the immediate basin of $\mathcal{A}_j.$  
We choose an arbitrary constant $C>0$ and a point $q\in{\bf v_j}\cap\mathcal{P}_j$. Then there exists a point $z_0\in \Omega_j$ so that for any $\tilde{q}\in Q:= (\cup_{l=0}^{\infty}f^{-l}(f^k(q)))\cap\Omega_j~(l, k$ are non-negative integers), the Kobayashi distance $d_{\Omega_j}(z_0, \tilde{q})\geq C$, where $d_{\Omega_j}$ is the Kobayashi metric. 
\end{thmB}

\begin{proof}

	Let $S:=\{z=re^{i\theta}, r>0, 0<\theta<\frac{2\pi}{m}\}$ be the sector with angle $2\pi/m$, including the attracting petal $\mathcal{P}_j$, then the angle between ${\bf v_j}$ and $\R^+$ is $\frac{\pi}{m}$. 
	We denote the boundary rays of $S$ by $l^1_{\R^+}:=\{z=r>0\}$ and $l^2_{\R^+}:=\{z=re^{i\frac{2\pi}{m}}, r>0\}$. 
	We choose $\theta_0>0$ and two rays $l_1:=\{z=r e^{i \theta_0}, r>0\}, l_2:=\{z=r e^{i (\frac{2\pi}{m}-\theta_0)}, r>0\}$. Then we denote by $T:=\{z=r e^{i \theta}, r>0, \theta_0<\theta<\frac{2\pi}{m}-\theta_0\}$ a sector inside $S.$ 
	
	Let $\phi(z): S\rightarrow \mathbb H$ with $\phi(z)=z^{m/2}$. Then the Kobayashi metric $F_{S}$ is $$F_S=\frac{|\phi'|}{\Im \phi}|dz|=\frac{m}{2r\sin(\frac{\theta m}{2})}|dz|.$$
	As in the proof of Theorem A, 
	we can choose two analogous ''Pac-Man'' $D^m_{R'_0}:=\{z=re^{i\theta}, 0<r<R'_0, \theta_0<\theta<\frac{2\pi}{m}-\theta_0\}, D^m_{R_0}:=\{z=re^{i\theta}, 0<r<R_0, \theta_0<\theta<\frac{2\pi}{m}-\theta_0\}$ central at $0$ with radius $R'_0, R_0 >0 ( R'_0< R_0)$, respectively, such that $D^m_{R_0} \subsetneqq T\cap \Omega_j$ and $f^n(D^m_{R'_0})\subset D^m_{R_0}.$ 
	Then similarly, we need to estimate the three Kobayashi distances from $z_0$ to any point $\tilde{z}\in\partial D^m_{R_0}$ (see the blue curve in Figure \ref{Figure6}), $z'\in {\bf v_j}$ (see the pink curve in Figure \ref{Figure6}),  and $\hat{z}\in\Omega_j\cap \big\{S':=\{z=re^{i\theta}, r>0, 0<\theta<\theta_0\}\big\}$ (see the green curve in Figure \ref{Figure6}),  and show that all of them are not less than $C$.
\begin{figure}[!htb]
	\centering 
	\includegraphics[width=0.5\textwidth,height=0.3\textheight]{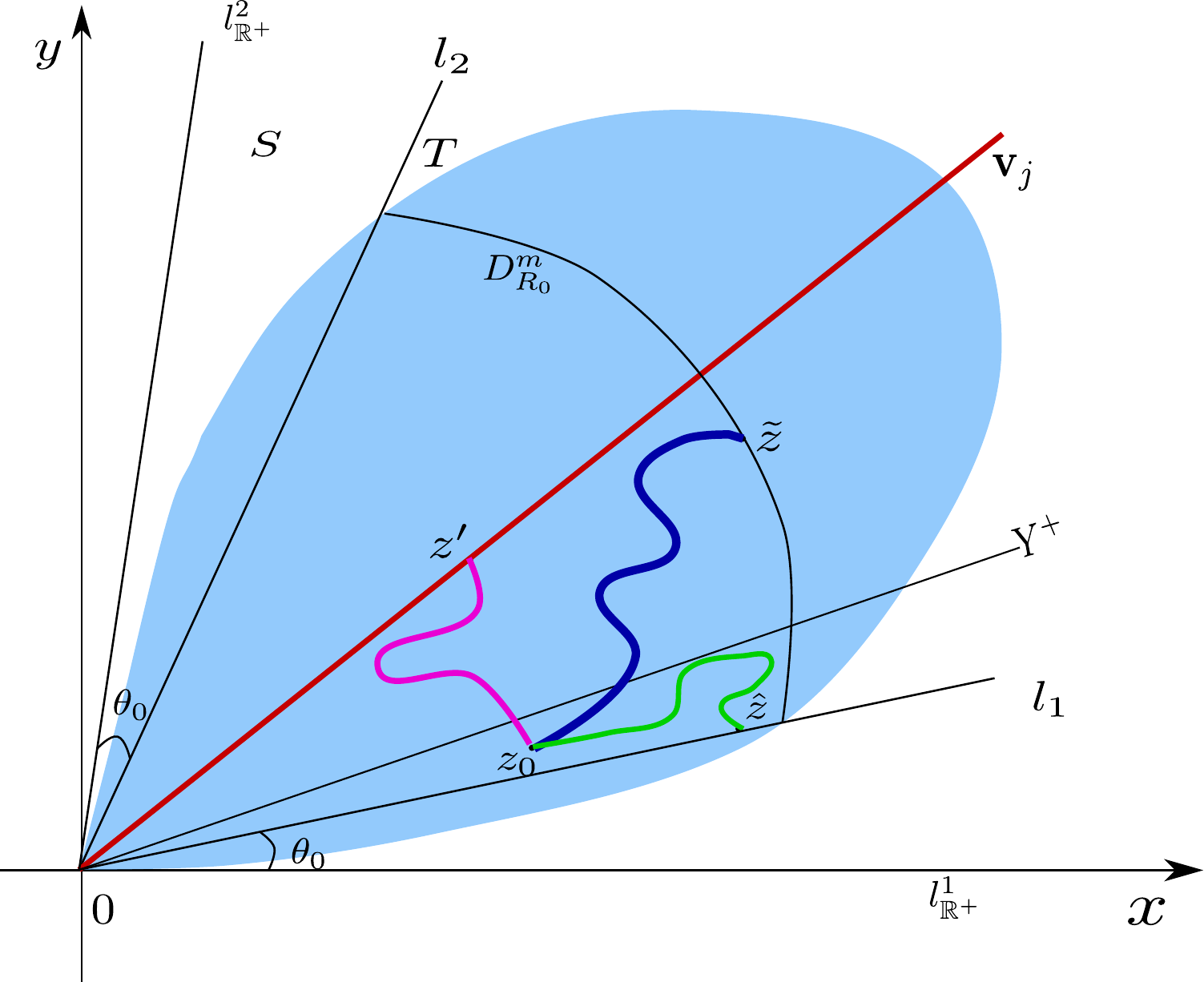}
	\caption{ The three Kobayashi distances in $\Omega_j$}
	\label{Figure6}
\end{figure}

	First, suppose $z_0\in D^m_\epsilon:=\{z=re^{i\theta}, 0<r<\epsilon, \theta_0<\theta<\frac{2\pi}{m}-\theta_0\}, \epsilon<<R'_0.$ Let us estimate the Kobayashi distance from $z_0$ to any point $\tilde{z}$ on the boundary of $D^m_{R_0}$, and we denote this distance by $d^1_{\Omega_j}(z_0, \tilde{z}).$
	
	\begin{equation*}
		\begin{aligned}
			d^1_{\Omega_j}(z_0, \tilde{z})
			&\geq d_{S}(z_0, \tilde{z})
			=\inf\int_{\gamma(t)}F_{S}(\gamma(t))|\gamma'(t)|dt\\
			&=\inf\int_{\gamma(t)}\frac{m}{2|\gamma(t)|\sin\frac{m\arg (\gamma(t))}{2}}|\gamma'(t)|dt\geq \inf\int_{\gamma(t)}\frac{m}{2|\gamma(t)|}|\gamma'(t)|dt\\
			&=\frac{m}{2}\inf\int_{\epsilon}^{R_0}\frac{|dr|}{r}\geq\frac{m}{2}(\ln R_0-\ln \epsilon),
		\end{aligned}	
	\end{equation*}
	where $\gamma(t)$ is a smooth path joining $z_0$ to $\tilde{z}$. 
	In addition, we can see that 
	$ d^1_{\Omega_j}(z_0, \tilde{z})\rightarrow\infty$ as $\epsilon\rightarrow0.$
	
	Second, we calculate the Kobayashi distance from $z_0$ to any point $z'\in {\bf v_j}$ denoted by $d^2_{\Omega_j}(z_0, z')$. Let $\epsilon=\epsilon_0$, i.e., fix $\epsilon,$ and $D^m_1\subset T$ be a scaling of $D^m_{\epsilon_0}$ by $S(z)=\frac{z}{|z_0|}$, sending $z_0, z'$ to $\tilde{z}_0:=\frac{z_0}{|z_0|}, \frac{z'}{|z_0|}$, respectively.
	By homogeneity, we know the Kobayashi distance  $d_{S}(z_0, z')=d_{S}(\tilde{z}_0, z'/|z_0|)$. 
	Since we hope to prove $d^2_{\Omega_j}(z_0, z')>C$, we need $z_0$ to be far from ${\bf v_j}$, and so does $\tilde{z}_0$. 
	Let $S_T:=\{z=e^{i\theta},  \theta_0<\theta<\frac{\pi}{2m}-\theta_0\}$, assume $\tilde{z}_0\in S_T$ and $\Re \tilde{z}_0$ is sufficiently big, then any curve from $\tilde{z}_0$ to $\frac{z'}{|z_0|}$ must pass through a point $\tilde{z}'$ on the ray $Y^+:=\{re^{i \frac{\pi}{2m}}, r>0\}.$
	Hence $d^2_{\Omega_j}(z_0,   z')\geq d_{S}(z_0, z')=d_{S}(\tilde{z}_0, z'/|z_0|)\geq d_{S}(\tilde{z}_0, \tilde{z}').$ Then we have
	\begin{equation*}
		\begin{aligned}
			d^2_{\Omega_j}(z_0, z')
			&\geq d_{S}(\tilde{z}_0, \tilde{z}')
			=\inf\int_{\tilde{z}_0}^{\tilde{z}'}F_{S}(z)\\
			&= \inf\int_{\tilde{z}_0}^{\tilde{z}'}\frac{m|dz|}{2r\sin(m\theta/2)}\\ 
			&= \inf\int_{\tilde{z}_0}^{\tilde{z}'}\frac{\frac{m\theta}{2}}{\sin(\frac{m\theta}{2})}\cdot\frac{m|dz|}{2r(m\theta/2)}\\
			&= c_1\inf\int_{\tilde{z}_0}^{\tilde{z}'}\frac{|dz|}{r\theta}\\
			&= c_1\inf\int_{\tilde{z}_0}^{\tilde{z}'}\frac{\sin\theta}{\theta}\cdot\frac{|dz|}{r\sin\theta}\\
			&=c_1 c_2\inf\int_{\tilde{z}_0}^{\tilde{z}'}\frac{|dz|}{r\sin\theta}\\
			&=c_1 c_2\inf\int_{\tilde{z}_0}^{\tilde{z}'}\frac{|dz|}{\Im z}\\
			&\geq c_1c_2| \ln (\Im \tilde{z}')-\ln (\Im \tilde{z}_0)|.\\
		\end{aligned}	
	\end{equation*}
where $c_1:=\inf\frac{\frac{m\theta}{2}}{\sin \frac{m\theta}{2}},  c_2:=\inf\frac{\sin\theta}{\theta}.$

Then there are three situations for the Kobayashi distance between $\tilde{z}_0$ and $\tilde{z}':$	

1) If $ \Im \tilde{z}'\geq e^C|\Im \tilde{z}_0|$ or $\Im \tilde{z}'\leq \frac{|\Im \tilde{z}_0|}{e^C}$ for some constant $C>1$, then $|\ln (\Im \tilde{z}')-\ln (\Im \tilde{z}_0)|\geq C,$ hence $d^2_{\Omega_j}(z_0, z')\geq C$ is true. 

2) lf  $\tilde{z}'\in\mathcal{L}:=\{z=x+iy, \frac{|\Im \tilde{z}_0|}{e^C}<y<e^C|\Im \tilde{z}_0|\}$. We need to prove that $d_{\Omega_j}(\tilde{z}_0, \tilde{z}')\geq c_1c_2\inf\int_{\tilde{z}_0}^{\tilde{z}'}\frac{|dz|}{\Im z} \geq C$
for $z\in\mathcal{L}$. 

Let $z=x+iy\in \mathcal{L}$, 
then $\Im z\leq e^C|\Im \tilde{z}_0|$. Hence we have  
$$\int_{\tilde{z}_0}^{\tilde{z}'}\frac{|dz|}{\Im z}
\geq\int_{\tilde{z}_0}^{\tilde{z}'} \frac{|dx|}{\Im z}\geq \int_{\tilde{z}_0}^{\tilde{z}'}\frac{|dx|}{e^C|\Im \tilde{z}_0| }=\frac{|\Re\tilde{z}'-\Re \tilde{z}_0|}{e^C}\frac{1}{|\Im \tilde{z}_0|}>\frac{1}{2e^C |\Im \tilde{z}_0|},$$
Note that the last inequality holds since we choose $\Re  \tilde{z}_0$ sufficiently big so that $\tilde{z}_0$ is close to $l_1$, and $|\Re\tilde{z}'-\Re \tilde{z}_0|>1/2$ because $\tilde{z}'$ will have to be close to $0$ since it lies on $Y^+\cap\mathcal{L}$ and its imaginary part is close to $0$, which makes $d^2_{\Omega_j}(z_0, z')$ as big as we want.

In other words, as long as  $\theta_0<\frac{1}{2Ce^C}$, we have $\Im \tilde{z}_0<\frac{1}{2Ce^C}$. In addition, $|\Re\tilde{z}'-\Re \tilde{z}_0|>1/2$, we obtain  $\int_{\tilde{z}_0}^{\tilde{z}'}\frac{|dz|}{\Im z}>C.$

3) If $\tilde{z}'\in  \mathcal{L}, $ but the curve $\gamma$ between $\tilde{z}_0$ and $\tilde{z}'$ get outside of $\mathcal{L}$ starting at some point $\tilde{z}''\in \gamma\cap\mathcal{L}$  for a while, then enter back to $\mathcal{L}$ again, then we still have $d^2_{\Omega_j}(z_0, z')\geq C$ is true because $d_S(z_0, z')\geq d_S(\tilde{z}_0, \tilde{z}'')\geq C.$ The last inequality holds since 1) is valid. We have a conclusion as same as  $d^2_{\mathcal{A}}(z_0, z')\geq C$ in the proof of Theorem A.

		At last, we estimate the Kobayashi distance from $z_0$ to any point  $\hat{z}\in\Omega_j\cap \big\{S':=\{z=re^{i\theta}, r>0, 0<\theta<\theta_0\}\big\}$ denoted by $d^3_{\Omega_j}(z_0, \hat{z}).$ We know  $d^3_{\Omega_j}(z_0, \hat{z})\geq d_S(z_0, \hat{z})=d_S(\tilde{z}_0, \hat{z}/|z_0|).$

	We use the method for computing $d^3_{\Omega_j}(z_0, \hat{z})$ as same as $d^3_{\mathcal{A}}(z_0, \hat{z})\geq C$ in the proof of Theorem A. We first choose a disk $\Delta^K_{S} (\tilde{z}_0, C)$ centered at $\tilde{z}_0$ with Kobayashi radius $C$. 
	Since this disk $\Delta^K_{S} (\tilde{z}_0, C)$ is a compact subset of $S$, there exists a sector $S'':=\{z=re^{i\theta}, r>0, \pi-\frac{\pi}{m}<\theta<\pi-\frac{\pi}{m}+\theta'_0\}$ such that $S''\cap\Delta^K_{S} (\tilde{z}_0, C)=\emptyset$, here we can assume that $\theta'_0<\theta_0.$ Therefore, $d_{S}(\tilde{z}_0, \hat{z}/|z_0|)\geq C$ for any $\hat{z}/|z_0|\in \Omega_j\cap S''.$ 
	Furthermore, for this new $\theta'_0,$ it does not change the conclusion of the estimation of $d^1_{\Omega_j}(z_0, \tilde{z})$ and $d^2_{\Omega_j}(z_0, z').$

\end{proof}

	\subsection{Dynamics inside the parabolic basin of $f(z)=z+az^{m+1}+(\text{higher terms}), m\geq1, a\neq0$}\label{subsec4}
	
	Finally, in this subsection, we sketch how to handle the behavior of orbits inside parabolic basins of general polynomials. Let us recall the statement of our main Theorem C:
	
	\begin{thmC}\label{the7}
	Let $f(z)=z+az^{m+1}+(\text{higher terms}), m\geq1, a\neq0,$ and $\Omega_j$ be the immediate basin of $\mathcal{A}_j.$ We choose an arbitrary constant $C>0$ and a point $q\in{\bf v_j}\cap\mathcal{P}_j$. Then there exists a point $z_0\in \Omega_j$ so that for any $\tilde{q}\in Q:= (\cup_{l=0}^{\infty}f^{-l}(f^k(q)))\cap\Omega_j~ (l, k$ are non-negative integers), the Kobayashi distance $d_{\Omega_j}(z_0, \tilde{q})\geq C$, where $d_{\Omega_j}$ is the Kobayashi metric. 
	\end{thmC}

The proof of Theorem B needs to be adjusted. To simplify the discussion, we consider the case $m=1:$
$$f(z)=z+z^2+(\text{higher terms}).$$

When there are no higher order terms, the crucial estimate of the Kobayashi metric comes from the fact that the parabolic basin is contained in $\mathbb{C}\setminus\R^+.$ Hence we could compare it with the Kobayashi metric on $\mathbb{C}\setminus\R^+.$

In the case of higher order terms, the parabolic basin might be more complicated. However, we can, instead of $\mathbb{C}\setminus\mathbb{R}^+$, use the double sheeted domain 
$$V_R:=\{z=re^{i\theta}, 0<r<R, -\theta_0<\theta<2\pi+\theta_0\}.$$

Next, we investigate the properties of $V_R$ to explain why we choose the double sheeted domain $V_R$ as above.

\begin{prop}\label{pro3}
	Let $\bar{D}_R:=\{z=re^{i\theta}, 0<r<R, -\theta_0<\theta<\theta_0\}$, $\mathcal{A}$ be the whole basin of $f, S_1 $ be the connected component of $\mathcal{A}\cap\bar{D}_R$ which contains $\{z=re^{i\theta_0}, 0<r<R\}$, and $S_2 $ be the connected component of $\mathcal{A}\cap\bar{D}_R$ which contains $\{z=re^{-i\theta_0}, 0<r<R\}.$ Then any two pieces $S_1, S_2 $ (see the left of Figure \ref{Figure7}) are disjoint in $\bar{D}_R$. 
\end{prop}
\begin{proof}
	\begin{figure}[!htb]
		\centering 
		\includegraphics[width=0.85\textwidth,height=0.2\textheight]{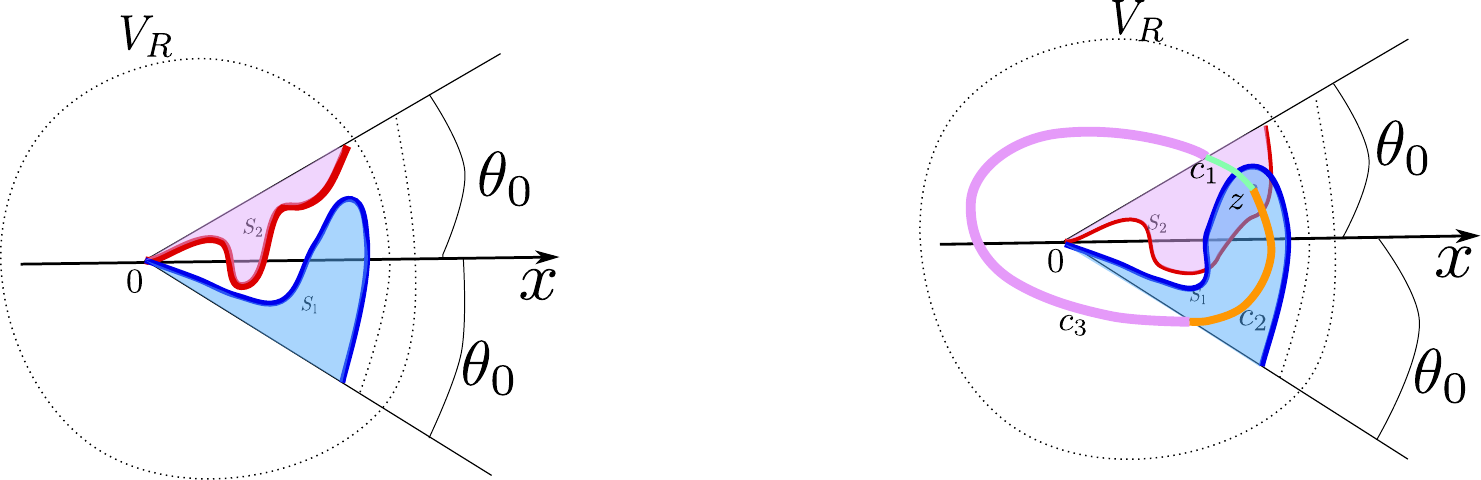}
		\caption{ Two pieces $S_1$ and $S_2 $ in $\bar{D}_R$}
		\label{Figure7}
	\end{figure} 
We know that, inside $V_R$ and near the origin, $\mathcal{A}$ contains the Left Pac-Man $D_{R}:=\{z=re^{i\theta}, 0<r<R, \theta_0<\theta<2\pi-\theta_0\}.$

If $S_1$ intersects $S_2$, then there is a point $z\in S_1\cap S_2$. We can draw three curves, $c_1$ from $z$ to $l_1:=\{z=re^{i\theta_0}\}$, $c_2$ from $z$ to $l_2:=\{z=re^{-i\theta_0}\}$, and $c_3\in D_R$ which connect $c_1$ and $c_2$. Hence $\mathcal{A}$ contains a closed curve $C:=c_1+c_2+c_3$ with the winding number $1$ around the origin (see the right of Figure \ref{Figure7}). 

We know that $f^n(z)\rightarrow0$ when $z\in C,$ since $C\in\mathcal{A}.$ In addition, by the maximum principle, we have $f^n(z)\rightarrow0$ when $z$ is inside the domain bounded by $C$. Hence $\mathcal{A}$ contains a neighborhood of $0,$ then $0$ is an attracting fixed point. However, this contradicts that $0$ is a parabolic fixed point of $f$.

\end{proof}
By Proposition \ref{pro3}, we can use the Kobayashi metric on $V_R$ instead of $\mathbb{C}\setminus\R^+$.

First, we know that $V_R$ can be mapped to a sector $S:=\{z=re^{i\theta}, r>0, \theta_1<\theta<\theta_2, \theta_1<<\theta_2<\pi/2\}$ by $\phi_1(z)=z^{c/2}$ when $c$ is sufficiently small.
Second, we can change $c$ such that $\theta_2=\pi-\theta_1$ by some map $\phi_2(z).$  At last, by some rotation map $\phi_3$, we can map $S$ to the upper half plane $\mathbb{H}$ (see the Figure \ref{Figure8}). 
\begin{figure}[!htb]
	\centering 
	\includegraphics[width=0.75\textwidth,height=0.2\textheight]{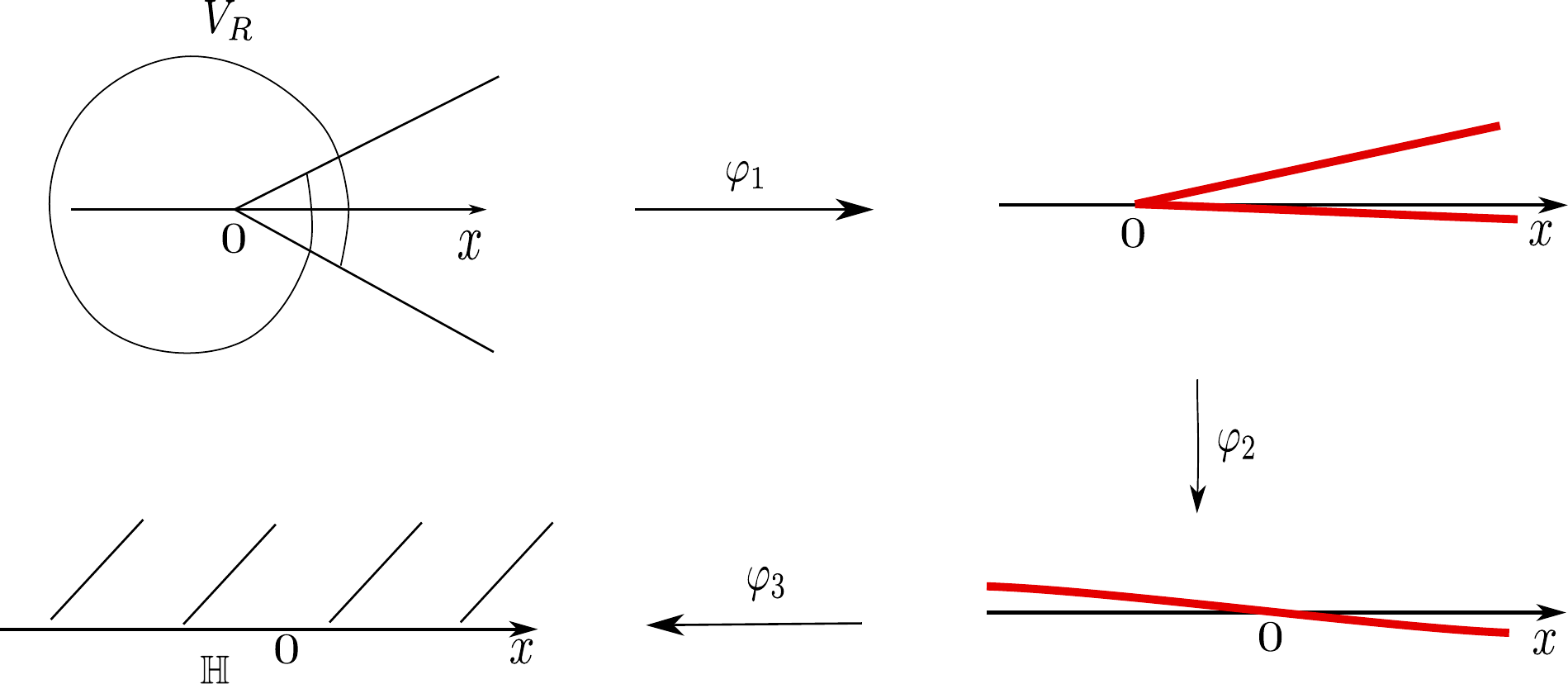}
	\caption{The maps $\phi_1, \phi_2$ and $\phi_3$}
	\label{Figure8}
\end{figure} 

Therefore, the map $\phi(z):=\phi_3\circ\phi_2\circ\phi_1$ from $V$ to the upper half plane $\mathbb{H}$ becomes $\phi(z)=e^{i\psi}z^{\frac{c}{2}}$, instead of $\phi(z)=z^{1/2}$, where $c$ is very close to $1$. Then, with the above setting, the rest of the estimation goes through as in Theorem B. 

If $m>1,$ it is difficult to draw the specific parabolic basins of $f$ or the attracting petals. Hence it is difficult to calculate the Kobayashi metric on the parabolic basin. Here we sketch the idea of how to prove this theorem for $m>1.$ We use Figure \ref{Figure9} to illustrate how we can choose 
 $V_R$ (see the domain with pink curves as its argument), and similarly, we have the same properties of $S_1, S_2$ as in Proposition \ref{pro3} (see the red curve and blue curve on Figure \ref{Figure9}). 
\begin{figure}[!htb]
	\centering 
	\includegraphics[width=0.75\textwidth,height=0.25\textheight]{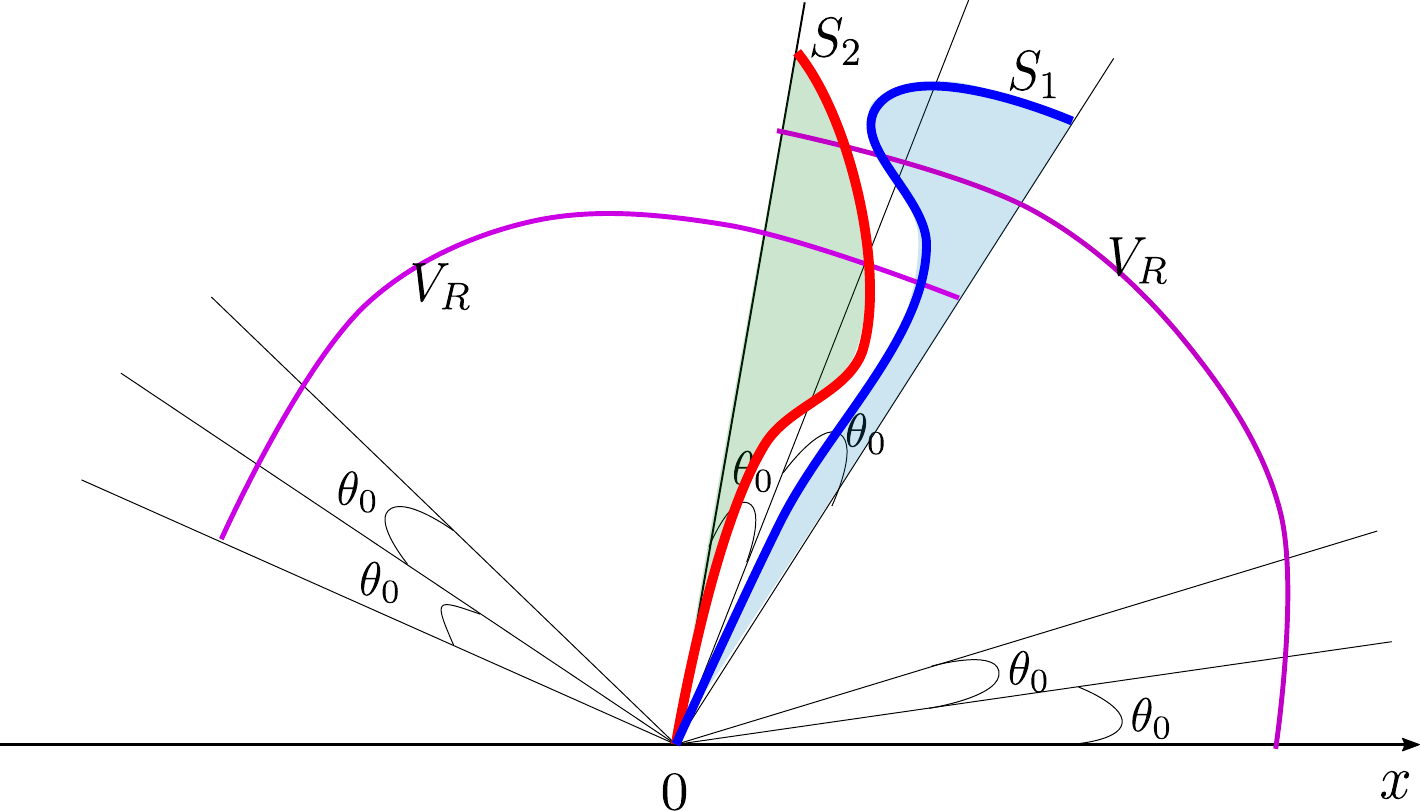}
	\caption{Two pieces $S_1$ and $S_2 $ }
	\label{Figure9}
\end{figure} 
Then the rest of the estimation goes through as when $m=1$. Thus, we are done.

\begin{corD}\label{co2}
Let $\Omega_{i,j}$ be the connected components of $\mathcal{A}_j$. Let $X\subset\mathcal{A}_j$ be the set of all $z_0\in\mathcal{A}_j$ so that $d_{\Omega_{i, j}}(z_0, \tilde{q})\geq C$ for any $\tilde{q}$ in the same connected component $\Omega_{i,j}$ as $z_0$. If $z\in X,$ then any point $w\in f^{-1}(z)$ is in $ X.$ Hence $X$ is dense in the boundary of $\mathcal{A}_j$.
\end{corD}
\begin{proof}
	By Theorem C, we know that $X\neq  \emptyset.$ Suppose $w\notin X,$ then there exists $ \tilde{q}\in Q$ so that $d_{\Omega_{i, j}}(w, \tilde{q})< C$. Since $f$ is distance decreasing, see Proposition \ref{pro2}, we have $d_{\Omega_{i, j}}(f(w), f(\tilde{q}))=d_{\Omega_{i,j}}(z, f(\tilde{q}))< C.$ This contradicts $z\in X.$

	Furthermore, we know that $\{f^{-n}(z)\}$ clusters at every point in Julia set. In particular, this is true if $z\in X.$ More precisely, $\{f^{-n}(z)\}$ equidistributes toward the Green measure. Therefore, $X$ is dense in the boundary of $\mathcal{A}_j$.
\end{proof}

\end{document}